\documentclass[12pt,oneside]{article}
\usepackage{amssymb,amsmath,amsfonts,amsbsy, amsthm, xspace, latexsym, amscd, graphicx, fancyhdr,setspace,url}
\usepackage[margin=1in]{geometry}
\usepackage{fancyhdr}
\usepackage{enumitem}

\swapnumbers
\theoremstyle{definition}
\newtheorem{definition}{Definition}[section]
\newtheorem{remark}[definition]{Remark}
\newtheorem{example}[definition]{Example}

\newtheorem{theorem}[definition]{Theorem}

\newtheorem{lemma}[definition]{Lemma}
\newtheorem{corollary}[definition]{Corollary}

\newcommand{\ifff}{\emph{ i\,f\,f }\;} 

\newcommand{\card}{\ensuremath{{\rm{card}}}}

\newcommand{\ub}{\ensuremath{\mathcal{UB}}}

\newlist{def-enumerate}{enumerate}{2}
\newlist{thm-enumerate}{enumerate}{2}
\newlist{set-enum}{enumerate}{1}
\newlist{axiom-enum}{enumerate}{2}

\setlist[def-enumerate,1]{label=\textbf{\arabic*}.,leftmargin=*,labelindent=\parindent}
\setlist[def-enumerate,2]{label=\textbf{(\alph*)},leftmargin=*,labelindent=\parindent}
\setlist[thm-enumerate]{label=(\emph{\roman*}),leftmargin=*,labelindent=\parindent,align=right}
\setlist[set-enum]{label=\textbf{Set \arabic*},leftmargin=*}
\setlist[axiom-enum,1]{label=\textbf{Axiom \arabic*}, leftmargin=*}
\setlist[axiom-enum,2]{label=\textbf{Axiom (\theenumi.\arabic*)}, leftmargin=*}

\begin{document}
\title{Another Proof of the Existence a Dedekind Complete Totally Ordered Field}
\author{James F. Hall\\ 
                        Mathematics Department\\                
                        California Polytechnic State University\\
                        San Luis Obispo, California 93407, USA\\\																				(james@slohall.com)\\
and\\
Todor D. Todorov\\ 
                        Mathematics Department\\                
                        California Polytechnic State University\\
                        San Luis Obispo, California 93407, USA\\
																							(ttodorov@calpoly.edu)
}
\date{}
\maketitle
\begin{abstract}

	We describe the Dedekind cuts explicitly in terms of non-standard rational numbers. This leads to another construction of a Dedekind complete totally ordered field or, equivalently, to another proof of the consistency of the axioms of the real numbers. We believe that our construction is simpler and shorter than the classical Dedekind construction and Cantor construction of such fields assuming some basic familiarity with non-standard analysis.
\end{abstract}
\section{Preliminaries: Ordered Fields and Infinitesimals}\label{S: Preliminaries: Orderable Fields and Infinitesimals}

	We recall the main definitions and properties of totally ordered rings and fields. We also recall the basic properties of infinitesimal, finite and infinitely large elements of such fields. For more details and for the missing proofs, we refer the reader to (Lang~\cite{lang}, Chapter XI),  (van der Waerden~\cite{waerden}, Chapter 11) and Ribenboim~\cite{riben}.

\begin{definition}[Orderable Ring]\label{D: Orderable Ring} Let $\mathbb{K}$ be a ring (field). Then:
\begin{enumerate}
  \item $\mathbb K$ is called \emph{orderable} if there exists a non-empty set $\mathbb{K}_+ \subset\mathbb{K}$ such that: (a) $0\not\in\mathbb{K}_+$; (b) $\mathbb{K}_+$ is closed under the addition and multiplication in $\mathbb{K}$; (c) For every $x\in\mathbb{K}$ either $x=0 \emph{ or } x\in\mathbb{K}_+\emph{ or } -x\in\mathbb{K}_+$.

\item A ring (field) $\mathbb K$ is \emph{formally real} if, for every $n\in\mathbb N$, the equation $\sum_{k=0}^n x_k^2 = 0$  in $\mathbb{K}$ has only the trivial solution $x_1=\dots=x_n=0$.
\end{enumerate}
\end{definition}

\begin{theorem}\label{L: orderable formally real}
  A field $\mathbb K$ is orderable \ifff $\mathbb K$ is formally real.
\end{theorem}
\begin{proof} For the proof we refer the reader to (van der Waerden~\cite{waerden}, Chapter 11).
\end{proof}

The field of complex numbers $\mathbb{C}$ is non-orderable, because the equation $x^2+y^2=0$ has a non-trivial solution, $x=i, y=1$, in $\mathbb{C}$. The field of the real p-adic numbers $\mathbb Q_p$ is non-orderable for similar reasons (see Ribenboim~\cite{riben} p.144-145).

\begin{definition}[Totally Ordered Rings]\label{D: Totally Ordered Rings} Let $\mathbb K$ be an orderable ring (field) and let the set $\mathbb{K}_+\subset\mathbb{K}$ satisfy the properties given above. Then:
\begin{enumerate}
\item  We define the relation $<_{\mathbb{K}_+}$ on $\mathbb{K}$ by $x<_{\mathbb{K}_+}y$ if $y-x\in\mathbb{K}_+$. We shall often write simply $<$ instead of $<_{\mathbb{K}_+}$ if the choice of $\mathbb{K}_+$ is clear from the context. Then $(\mathbb{K}, +, \cdot,  <)$, denoted for short by $\mathbb{K}$, is called a {\em totally ordered ring (field)}.

\item  Let $\mathbb{K}$ be a totally ordered ring (field) and let $x\in\mathbb{K}$ and $A\subset \mathbb K$. (a) We define the {\em absolute value} of $x$ by $|x|=:\max(-x,x)$; (b) We denote by  $\ub(A)$ the {\em set of the upper bounds} for $A$; (c) We denote by $\sup_\mathbb{K}(A)$ or, simply by $\sup(A)$, the {\em least upper bound} for $A$ (if exists).

\item A totally ordered ring (field) $\mathbb{K}$ is called \emph{Archimedean} if for every $x\in\mathbb{K}$, there exists $n\in\mathbb{N}$ such that $|x|\leq n$.

\item A totally ordered set (ring or field)  $\mathbb{K}$ is \emph{Dedekind complete} if every non-empty subset of $\mathbb{K}$ that is bounded from above has a supremum. 
\end{enumerate}
\end{definition}

\begin{theorem}[Rationals and Irrationals]\label{T: Rationals and Irrationals} Let $\mathbb{K}$ be a totally ordered field. Then:
\begin{description}

\item (i) $\mathbb{K}$ contains a copy of the field of the rational numbers $\mathbb{Q}$ under the order field embedding  $\sigma:\mathbb Q\to\mathbb K$ defined by: $\sigma(0)=:0$, $\sigma(n)=:n\cdot1$, $\sigma(-n)=:-\sigma(n)$ and $\sigma(\frac{m}{k})=:\frac{\sigma(m)}{\sigma(k)}$ for $n\in\mathbb N$ and $m,k\in\mathbb Z$. We shall simply write $\mathbb{Q}\subseteq\mathbb{K}$ for short.  
\item (ii) If $\mathbb{K}\setminus\mathbb{Q}$ is non-empty, then $\mathbb{K}\setminus\mathbb{Q}$ is dense in $\mathbb{K}$ in the sense that for every $a, b\in\mathbb{K}$, such that $a<b$, there exists $x\in\mathbb{K}\setminus\mathbb{Q}$ such that $a<x<b$.
\item (iii) If $\mathbb{K}$ is Archimedean, then $\mathbb{Q}$ is also dense in $\mathbb{K}$ in the sense that for every $a, b\in\mathbb{K}$ such that $a<b$ there exists $q\in\mathbb{Q}$ such that $a<q<b$.
\end{description}
\end{theorem}


\begin{definition}[Infinitesimal, Finite and Infinitely Large]\label{D: Infinitesimal, Finite and Infinitely Large}
  Let $\mathbb{K}$ be a totally ordered field. We define:
\begin{align}
& \mathcal{I}(\mathbb{K})=:\{x\in\mathbb{K} : |x|< 1/n \text{ for all } n\in\mathbb N\},\notag\\
    &\mathcal{F}(\mathbb{K})=:\{x \in\mathbb{K} : |x|\le n  \text{ for some } n\in\mathbb N\},\notag\\
   &\mathcal{L}(\mathbb{K})=:\{x \in\mathbb{K} : \mathbb{N})(n<|x|  \text{ for all } n\in\mathbb N\}.
\end{align}

  The elements in $\mathcal I(\mathbb K), \mathcal F(\mathbb K), \textrm{ and } \mathcal L(\mathbb K)$ are referred to as \emph{infinitesimal (infinitely small), finite and infinitely large}, respectively. We sometimes write $x\approx0$ if $x\in\mathcal I(\mathbb K)$ and $x\approx y$ if $x-y\approx 0$, in which case we say that $x$ is \emph{infinitesimally close} to $y$.
\end{definition}

	The next result follows directly from the above definition.

\begin{lemma} Let $\mathbb K$ be a totally ordered ring. Then: (a) $\mathcal I(\mathbb K)\subset \mathcal F(\mathbb K)$; (b) $\mathbb K=\mathcal F(\mathbb K)\cup\mathcal L(\mathbb K)$;
(c) $\mathcal F(\mathbb K)\cap\mathcal L(\mathbb K)=\varnothing$. If $\mathbb{K}$ is a field, then: (d) $x\in\mathcal I(\mathbb K)$ \ifff $\frac{1}{x}\in\mathcal L(\mathbb K)$ for every non-zero $x\in \mathbb K$.
\end{lemma}

\begin{theorem}\label{T: Maximal Ideal}
  Let $\mathbb{K}$ be a totally ordered field. Then $\mathcal{F}(\mathbb{K})$ is an Archimedean ring and $\mathcal{I}(\mathbb{K})$ is a maximal ideal of $\mathcal{F}(\mathbb{K})$. Moreover, $\mathcal{I}(\mathbb{K})$ is a \emph{convex ideal} in the sense that $a\in\mathcal{F}(\mathbb{K})$ and $|a|\le b\in\mathcal{I}(\mathbb{K})$ implies $a\in\mathcal{I}(\mathbb{K})$. Consequently $\mathcal{F}(\mathbb{K})/\mathcal{I}(\mathbb{K})$ is a totally ordered Archimedean field.
\end{theorem}     

	Here are some familiar properties of ordered fields expressed ``in terms of infinitesimals''.

\begin{theorem}[Archimedean Property]\label{T: Archimedean Property}
  Let $\mathbb{K}$ be a totally ordered ring. Then the following are equivalent: (i) $\mathbb{K}$ is Archimedean. (ii) $\mathcal{F}(\mathbb{K})=\mathbb{K}$. (iii) $\mathcal L(\mathbb K)=\varnothing$. If $\mathbb{K}$ is a field, then each of the above is also equivalent to $\mathcal I(\mathbb K)=\{0\}$.  
\end{theorem}

  Notice that Archimedean rings (which are not fields) might have non-zero infinitesimals. Indeed,  if $\mathbb K$ is a non-Archimedean field, then $\mathcal F(\mathbb K)$ is always an Archimedean ring, but it has non-zero infinitesimals (see Example~\ref{Ex: Field of Rational Functions} below).

\begin{corollary}\label{L: order -> arch} Every Dedekind complete totally ordered field is Archimedean.
\end{corollary}

\begin{proof}
  Let $\mathbb{D}$ be a Dedekind complete totally ordered field and suppose (to the contrary) that $\mathbb{D}$ is non-Archimedean. Then  $\mathcal L(\mathbb{D})\not=\varnothing$ by Theorem~\ref{T: Archimedean Property}. Thus $\mathbb{N}\subset\mathbb{D}$ is bounded from above by $|\lambda|$ for any $\lambda\in\mathcal L(\mathbb D)$. Let $\alpha\in\mathbb{K}$ be the least upper bound of $\mathbb{N}$. Then there exists $n\in\mathbb{N}$ such that $\alpha-1< n$ implying $\alpha< n+1$, a contradiction.
\end{proof}

\begin{example}[Field of Rational Functions]\label{Ex: Field of Rational Functions} Let $\mathbb{R}(t)$ be the field of rational functions in one variable with real coefficients. We supply $\mathbb{R}(t)$ with an ordering given by $f< g$ in $\mathbb{R}(t)$ if there exists $n\in\mathbb{N}$ such that $g(t)-f(t)>0$ in $\mathbb{R}$ for all $t\in (0, 1/n)$.  Notice that $\mathbb{R}(t)$ is a non-Archimedean field: $t, t^2, t+t^2$, etc. are positive infinitesimals, $1+t, 2+t^2, 3+t+t^2$, etc. are finite, but non-infinitesimal, and $1/t, 1/t^2, 1/(t+t^2)$, etc. are infinitely large elements of  $\mathbb{R}(t)$. 
\end{example}
\section{Existence of a Dedekind Complete Field}\label{S: Existence of a Dedekind Complete Field}

	In this section we present a proof of the existence of a Dedekind compete totally ordered field based on the saturation principle from non-standard analysis (Lindstr\o m~\cite{lindstrom}, p. 49). For another non-standard proof based on the concurrence theorem, we refer to (Davis~\cite{davis}, p. 53). For the classical proofs of the same result due to Dedekind and Cantor, we refer to  Rudin~\cite{poma} and Hewitt \& Stromberg~\cite{hewitt}, respectively. We also mention a more recent proof of the existence of a Dedekind complete field in Banaschewski~\cite{bana}, based on a form of completeness introduced by Hilbert. 

\begin{enumerate}

\item Let $^*\mathbb Q$ be the non-standard extension of the  field of the rational numbers $\mathbb Q$ in a polysaturated non-standard model with a set of individuals $\mathbb Q$ (Lindstr\o m~\cite{lindstrom}, p. 51). Let $\mathcal{F}(^*\mathbb Q)$ and $\mathcal{I}(^*\mathbb Q)$ denote the sets of the finite and infinitesimal elements of $^*\mathbb Q$, respectively (Definition~\ref{D: Infinitesimal, Finite and Infinitely Large}). Notice that $\card(\mathcal{F}(^*\mathbb Q))/\mathcal{I}(^*\mathbb Q))<\kappa$ since $\mathcal{F}(^*\mathbb Q)/\mathcal{I}(^*\mathbb Q)$ is an Archimedean field by Theorem~\ref{T: Maximal Ideal}. We denote by $\chi: \mathcal{F}(^*\mathbb Q)\to\mathcal{F}(^*\mathbb Q)/\mathcal{I}(^*\mathbb Q)$ the canonical homomorphism. We also let $\mathcal{A}_+=\{\alpha\in\mathcal{F}(^*\mathbb Q): \alpha>0 \text{ and } \alpha\not\approx 0\}$. 

\item We define $\mathbb D =\{C_\alpha: \alpha\in\mathcal{F}(^*\mathbb Q) \}$, where $C_\alpha=\{q\in\mathbb Q: q<\alpha\}$. We observe that $C_\alpha = C_\beta$ \ifff $\alpha\approx\beta$. Consequently, $\card(\mathbb D)<\kappa$. We convert $\mathbb{D}$ into a totally ordered set (chain) under the ordering: $C_\alpha<C_\beta$ if  $\alpha<\beta$ and $\alpha\not\approx\beta$ (or if $\beta-\alpha\in\mathcal A_+$, for short). We also observe that the mapping $q\to C_q$ from $\mathbb Q$ into $\mathbb D$ is an injection which preserves the usual order in $\mathbb Q$. We define a monad-like function $M: \mathbb D\to \mathcal P(^*\mathbb Q)$ by $M(x)=\{\alpha\in\mathcal F(^*\mathbb Q): C_\alpha=x\}$. Here $\mathcal P(^*\mathbb Q)$ stands for the power set of $^*\mathbb Q$.

\begin{theorem} $\mathbb D$ is a Dedekind complete totally ordered set.
\end{theorem} 

\begin{proof} Let $A\subset\mathbb D$ be a set bounded from above by some $C\in\mathbb D$ and let $B=: U\!B(A)$ stand for the set of upper bounds of $A$ in $\mathbb D$. If $A$ is finite or $C\in\mathbb D$, then $\sup(A)$ exists since $\sup(A)=\max(A)$. Suppose that $A$ is an infinite set and $C\notin A$ and observe that the family of open intervals $\left\{(\alpha, \beta)\right\}_{C_\alpha\in A,\; C_\beta\in B}$ has the finite intersection property. Next, we have to select a subfamily of cardinality at most $\kappa$ which still has the finite intersection property. By the axiom of choice, there exists a function $\varphi: A\to{^*\mathbb Q}$ such that $\varphi(a)\in M(a)$ for all $a\in A$. Similarly, there exists  $\psi: B\to{^*\mathbb Q}$ such that $\psi(b)\in M(b)$ for all $a\in B$. The family $\left\{(\varphi(a), \psi(b))\right\}_{a\in A,\; b\in B}$ is of cardinality less than $\kappa$ and by the saturation principle (Lindstr\o m~\cite{lindstrom}, p. 49), there exists $\gamma\in \mathcal F(^*\mathbb Q)$ such that $\varphi(a)<\gamma< \psi(b)$ for all $a\in A$ and all $b\in B$. Thus $\sup(A)=C_\gamma$. 
\end{proof}

\item We define addition and multiplication in $\mathbb D$ by $C_\alpha+C_\beta=C_{\alpha+\beta}$ and $C_\alpha C_\beta=C_{\alpha\beta}$, respectively. We observe that $\mathbb D$ is a field with zero $0=:C_0$ and unit $1=:C_1$, and with additive inverse $-C_\alpha=C_{-\alpha}$ and multiplicative inverse $C_\alpha^{-1}=C_{\alpha^{-1}}$. 

\begin{theorem} Also $\mathbb D$ is a totally ordered field. 
\end{theorem}
\begin{proof} Let $\mathbb D_+=\{C_\alpha\in\mathbb D: C_\alpha>0 \}$. We have $C_0\not\in\mathbb D_+$ since $0\notin\mathcal A_+$. Also, $\mathbb D_+$ is closed under the addition and multiplication in $\mathbb D$ since $\mathcal{A}_+$ is closed under the addition and multiplication in $^*\mathbb Q$. Finally, suppose that $C_\alpha\not= 0$, i.e. $\alpha\not\approx 0$. We have either $\alpha>0$ or $\alpha<0$ by the trichotomy of the order in $^*\mathbb Q$. The latter implies either $C_\alpha>0$ or $C_\alpha<0$.
\end{proof}
\end{enumerate}
\begin{remark}[Uniqueness of the Operations] Recall that every Dedekind complete ordered set (chain) $\mathbb D$, which contains a copy of the ordered set $(\mathbb Q, <)$, determines {\em uniquely} the algebraic operations in $\mathbb D$ which convert $\mathbb D$ into a totally ordered field by the formulas: 
\begin{description}
\item[(i)] If $C_\alpha, C_\beta\in\mathbb D$, then  

$C_\alpha + C_\beta=\sup_\mathbb D\{p+q: p, q\in \mathbb Q,\;  p<\alpha \text{ and }q<\beta\}$; $-C_\alpha=\sup_\mathbb D\{q\in\mathbb Q: q<-\alpha\}$.

\item[(ii)] If $C_\alpha, C_\beta>0$, then $C_\alpha C_\beta=\sup_\mathbb D\{pq: p, q\in \mathbb Q_+,\;  p<\alpha \text{ and }q<\beta\}$ and  

	$C_\alpha^{-1}=\sup_\mathbb D\{q\in\mathbb Q_+: q<\alpha^{-1}\}$. 

\item [(iii)] We also let $(-C_\alpha)(-C_\beta)=C_\alpha C_\beta$ and $(-C_\alpha)C_\beta)=-(C_\alpha C_\beta)$.                                            

In fact, we do not need these formulas for our construction; our definitions of addition and multiplication in $\mathbb D$, presented in \# \textbf{3} above, are simpler. 
\end{description}
\end{remark}

\begin{remark}[Why polysaturation ?] The above arguments hold for any $\mathfrak{c}^+$-saturated non-standard model, where $\mathfrak{c}=\card(\mathbb R)$ (Lindstr\o m~\cite{lindstrom}, p. 49). We prefer to use a polysaturated non-standard model only to avoid even mentioning $\mathbb R$ or its cardinality $\mathfrak{c}$.
\end{remark}

\begin{remark}[Original Dedekind Cuts] Our result shows that $\mathbb D$ consists exactly of the usual (original) Dedekind cuts. Thus our construction can be viewed as a ``explicit parametric description of the Dedekind cuts in terms of $^*\mathbb Q$''.
\end{remark}
  
\begin{corollary}[Another Complete Field] $\mathcal{F}(^*\mathbb Q)/\mathcal{I}(^*\mathbb Q)$ is a Dedikind complete totally ordered field. 

\begin{proof} $\mathcal{F}(^*\mathbb Q)/\mathcal{I}(^*\mathbb Q)$ is a totally ordered field by Theorem~\ref{T: Maximal Ideal} and it is easy to see that the mapping $C_\alpha\to\chi(\alpha)$, from $\mathbb D$ to $\mathcal{F}(^*\mathbb Q)/\mathcal{I}(^*\mathbb Q)$, is an order field isomorphism. 
\end{proof}

 We should mention that a direct proof of the Dedekind completeness of $\mathcal{F}(^*\mathbb Q)/\mathcal{I}(^*\mathbb Q)$ based on the concurrence theorem appears in (Davis~\cite{davis}, p. 53).
\end{corollary}




\begin{thebibliography}{9}
 
  \bibitem{bana} B. Banaschewski. \emph{On Proving the Existence of Complete Ordered Fields.} American Math Monthly 105(6), pp. 548-551.
  \bibitem{davis} M. Davis. \emph{Applied Non-standard Analysis}. Dover, 2005.
  \bibitem{hewitt} E. Hewitt and K. Stromberg. \emph{Real and Abstract Analysis}. Springer, first edition, 1965.

  
  \bibitem{lang} S. Lang, Algebra, Addison-Wesley, 1992.
 
  \bibitem{lindstrom} T. Lindstr\o m. An invitation to nonstandard analysis. In: Cutland N (ed) \emph{Nonstandard Analysis and its applications}. Cambridge University Press, London, 1988, pp 1-105.


  \bibitem{riben} P. Ribenboim, {\em The Theory of Classical Valuation}, Springer Monographs in Mathematics, Springer-Verlag, 1999.
 
  \bibitem{poma} W. Rudin. \emph{Principles of Mathematical Analysis}. McGraw Hill, third edition, 1976.
 
  \bibitem{waerden} B. L.van der Waerden. \emph{Algebra}. Springer, seventh edition, 1966.
\end{thebibliography}
\end{document}